\documentclass{amsart}
\usepackage[margin=1.4in]{geometry}

\usepackage{amsfonts}

\usepackage{setspace}
\usepackage{graphicx}

\usepackage{hyperref}
\usepackage{graphicx}
\usepackage{amsmath, amsthm, amscd, amsfonts, amssymb, graphicx, color}
 \usepackage{url}
\usepackage{enumerate}
 \makeatletter
\let\reftagform@=\tagform@
\def\tagform@#1{\maketag@@@{(\ignorespaces\textcolor{blue}{#1}\unskip\@@italiccorr)}}
\renewcommand{\eqref}[1]{\textup{\reftagform@{\ref{#1}}}}
\makeatother
\usepackage{hyperref}
\hypersetup{colorlinks=true, linkcolor=red, anchorcolor=green,
citecolor=cyan, urlcolor=red, filecolor=magenta, pdftoolbar=true}

\usepackage{epsfig}        
\usepackage{epic,eepic}       

\newtheorem{theorem}{Theorem}
\theoremstyle{plain}

\newtheorem{conjecture}{Conjecture}
\newtheorem{corollary}{Corollary}

\newtheorem{definition}{Definition}

\newtheorem{remark}{Remark}

\numberwithin{equation}{section}

\begin{document}
\title[A note on $h$-convex functions]{A note on $h$-convex functions}
\author[M.W. Alomari]{Mohammad W. Alomari}

\address{Department of Mathematics, Faculty of Science and
Information Technology, Irbid National University, 2600 Irbid
21110, Jordan.} \email{mwomath@gmail.com}
\date{\today}
\subjclass[2000]{26A15, 26A16, 26A51}

\keywords{$h$-Convex function, H\"{o}lder continuous}

\begin{abstract}
In this work, we discuss the continuity of $h$-convex functions by introducing the concepts of $h$-convex curves ($h$-cord). Geometric interpretation of $h$-convexity is given. The fact that for a $h$-continuous function $f$, is being $h$-convex if and only if is $h$-midconvex is proved. Generally, we prove that if $f$ is $h$-convex then $f$ is $h$-continuous. A discussion regarding derivative characterization of $h$-convexity is also proposed.
\end{abstract}

\maketitle
\section{Introduction}

Let $I$ be a real interval. A function $f:I\to \mathbb{R}$ is
called convex iff
\begin{align}
f\left( {t\alpha +\left(1-t\right)\beta} \right)\le tf\left(
{\alpha} \right)+ \left(1-t\right) f\left( {\beta}
\right),\label{eq1.1}
\end{align}
for all points $\alpha,\beta \in I$ and all $t\in [0,1]$. If $-f$
is convex then we say that $f$ is concave. Moreover, if $f$ is
both convex and concave, then $f$ is said to be affine.

In 1979, Breckner \cite{B1} introduced the class of $s$-convex
functions (in the second sense), as follows:
\begin{definition}
\label{def1}  Let $I \subseteq \left[0,\infty\right)$ and $s\in
(0,1]$, a function $f : I\to \left[0,\infty\right)$ is
 $s$-convex function or that $f$ belongs to the
class $K^2_s\left(I\right)$ if for all $x,y \in I$ and $t?\in
[0,1]$ we have
\begin{align*}
f \left(tx+\left(1-t\right)y\right) \le
t^sf\left(x\right)+\left(1-t\right)^sf\left(y\right).
\end{align*}
\end{definition}
In the last years, among others, the notion of $s$-convex
functions is discriminated and starred. In literature a few papers
devoted to study this type of convexity. The building theories of
$s$-convexity as geometric and analytic tools are still under
consideration, development and examine. Due to Hudzik and
Maligranda (1994)\nocite{HM}, two senses of $s$-convexity $(0 < s
\le 1)$ of real-valued functions are known in the literature, and
given below:
\begin{definition}
\label{def2.3.7}A function $f: \mathbb{R}_+ \to \mathbb{R} $,
where $\mathbb{R}_+ = \left[{0,\infty} \right)$, is said to be
$s$-convex in the first sense if
\begin{eqnarray}
\label{eq2.3.1}f\left( {\alpha x + \beta y}\right) \le \alpha ^s
f\left( x \right) + {\beta}^s f\left( y \right)
\end{eqnarray}
for all $x, y \in \left[ {0,\infty } \right)$, $\alpha, \beta \ge
0$ with $\alpha^s + \beta^s = 1$ and for some fixed $s \in \left(
{0,1 } \right]$. This class of functions is denoted by $K_s^1$.
\end{definition}
\noindent This definition of $s$-convexity, for so called
$\varphi$-functions, was introduced by Orlicz in 1961 and was used
in the theory of Orlicz spaces. A function $f: \mathbb{R}_+ \to
\mathbb{R}_+$ is said to be a $\varphi$-function if $f(0) = 0$ and
$f$ is nondecreasing and continuous. The symbol $\varphi$ stands
for an Orlicz function, i.e., $\varphi$ is defined on the real
line $\mathbb{R}$ with values in $[0, +\infty]$ and is convex,
even, vanishing and continuous at zero. For further details see
\cite{HM}, \cite{J} ,\cite{MO},\cite{R}.

In fact, Breckner \cite{B1} walked in the footsteps of Orlicz's
definition (Definition \ref{def2.3.7}) and introduced another type
of $s$-convexity or what so called \emph{Breckner $s$-convex}, as
follows:
\begin{definition}
\label{Breckner}A function $f: \mathbb{R}^+ \to \mathbb{R}$, where
$\mathbb{R}^+ = \left[{0,\infty} \right)$, is said to be
$s$--convex in the second sense if
\begin{eqnarray}
\label{Breckner1}f\left( {\alpha x + \beta y}\right) \le \alpha ^s
f\left( x \right) + {\beta}^s f\left( y \right)
\end{eqnarray}
for all $x, y \in \left[ {0,\infty } \right)$, $\alpha, \beta \ge
0$ with $\alpha + \beta = 1$ and for some fixed $s \in \left( {0,1
} \right]$. This class of functions is denoted by $K_s^2$.
\end{definition}

\begin{remark}
We note that, it can be easily seen that for $s=1$, $s$-convexity
(in both senses) reduces to the ordinary convexity of functions
defined on $\left[ {0,\infty }\right) $.
\end{remark}

In general, a real-valued function $f$ defined on an open convex
subset $C$ of a linear space is called Breckner $s$-convex if
\eqref{Breckner1} holds for every $x,y \in C$, $\alpha,\beta  \in
[0,1]$ with $\alpha+\beta = 1$, where $s \in (0,1)$ is fixed. More
preciously,  Breckner considered an open convex subset
$\mathbb{M}$ of a linear space $\mathbb{L}$ and defined
$f:\mathbb{M} \subseteq \mathbb{L}\to \mathbb{R}$, to be
$s$-convex if \eqref{Breckner1} holds, for all $x,y \in
\mathbb{M}$, $\alpha,\beta \in [0,1]$ with $\alpha+\beta = 1$,
where $s \in (0,1)$ is fixed. Also, Breckner considered a special
case of $s$-convex functions which is so called rational
$s$-convex, that is for all rational $\alpha,\beta \in [0,1]$ with
$\alpha+\beta = 1$ and points $x,y \in \mathbb{M}$, the inequality
\eqref{Breckner1} holds. Furthermore, Breckner proved that for
locally bounded above $s$-convex functions defined on open subsets
of linear topological spaces are continuous and nonnegative.

In 1978,   Breckner and Orb\'{a}n \cite{B2} considered functions
from a convex subset of a real or complex Hausdorff topological
linear space of dimension greater than 1 into an ordered
topological linear space such that all its order-bounded subsets
are bounded, and showed that Breckner $s$-convex functions with $s
\in (0, 1]$ are continuous on the interior of their domain.

In 1994, Breckner \cite{B3}  (see also \cite{B4}) proved that for
a rationally $s$-convex function continuity and local
$s$-H\"{o}lder continuity are equivalent at each interior point of
the domain of definition of the function. Furthermore, it is shown
that a rationally $s$-convex function which is bounded on a
nonempty open convex set is $s$-H\"{o}lder continuous on every
compact subset of this set. Indeed, Breckner [2], showed that if a
real-valued function defined on a convex subset of a linear space
endowed with topology generated by a direct pseudonorm is
continuous and rationally Breckner $s$-convex for an $s \in (0,
1]$, then it is locally $s$-H\"{o}lder.

In 1994, Hudzik and Maligranda \cite{HM}, realized the importance
and undertook a systematic study of $s$-convex functions in both
sense. They compared the notion of Breckner $s$-convexity with a
similar one of \cite{MO}. A function $f$ is Orlicz $s$-convex if
the inequality (\ref{eq2.3.1}) is satisfied for all $\alpha,
\beta$ such that $\alpha^s + \beta^s = 1$. Hudzik and Maligranda,
among others, gave an example of a non-continuous Orlicz
$s$-convex function, which is not Breckner $s$-convex.\\

In 2001, Pycia \cite{P} established a direct proof of Breckner's
result that Breckner $s$-convex real-valued functions on finite
dimensional normed spaces are locally $s$-H\"{o}lder. The same  result was proved in \cite{mwo1} where  different context was considered. For the same result regarding convexity see \cite{C1} and \cite{C2}.\\

In the  2008, Pinheiro \cite{MRP1} studied the class of $K_s^1$ of
$s$-convex functions and explained why the first $s$-convexity
sense was abandoned by the literature in the field. In fact,
Pinheiro \nocite{MRP1}, proposed some criticisms to the current
way of presenting the definition of $s$-convex functions. We may
summarize Pinheiro criticisms in the following points:
\begin{enumerate}
    \item What is the `true' difference between convex and $s$-convex
    in both senses.

    \item So far, Pinheiro did not find references, in the literature,
    to the geometry of an $s$-convex function, what, once more, makes
    it less clear to understand the difference between an $s$-convex
    and a convex function whilst there are clear references to the
    geometry of the convex functions.
\end{enumerate}
In the same paper \cite{MRP1}, Pinheiro revised the class of
$s$--convexity in the first sense. In \cite{MRP2}, Pinheiro
proposed a geometric interpretation for this type of functions.
\begin{definition}
    \label{Pinheiro}Let $U$ be any subset of $[0,\infty)$. A function
    $f: X \to \mathbb{R}$, is said to be $s$--convex in the first
    sense if
    \begin{eqnarray}
    \label{Pinheiro1}f\left( {\lambda x + \left( {1-\lambda^s}
        \right)^{1/s} y}\right) \le \lambda^s f\left( x \right) + \left(
    {1-\lambda^s} \right) f\left( y \right)
    \end{eqnarray}
    for all $x, y \in U$ and $\lambda \in [0,1]$.
\end{definition}
The presented reason from Pinheiro to why $s$-convexity in the
first sense got abandoned in the literature, is that, if one takes
$x = y = \frac{1}{4}$ with $\alpha = \frac{1}{2}$ and $\beta = 1$
for example, one gets that $\alpha x + \beta y = 0.125 + 0.25 =
0.375$. So that, if $s = \frac{1}{2}$, then the value of $\alpha x
+ \beta y$ would lie outside of the interval $[x,y]$, on the
contrary of this, the value of $\alpha x + \beta y$ would lie
inside of the interval $[x,y]$ in case of convexity. With this the
first sense of $s$-convexity becomes a close to the meaning of
convexity and so the geometric explanation of $s$-convex function
is easy to be compared with the geometry of convex function if
some further restrictions are imposed to it.

\noindent The proposed geometric description for $s$-convex curve
in the first sense stated by  Pinheiro \cite{MRP1}--\cite{MRP6} as
follows:

\begin{definition}
    \label{def2.3.22} A function $f: X \subset \mathbb{R}_+ \to
    \mathbb{R}$  is called $s$-convex in the first sense if and only
    if one in two situations occur:

    \begin{itemize}
        \item $0 < s_1 < 1$, $f$ then belonging to $K^1_s$ , for $0 < s
        \le s_1$: The graph of f lies below (L), which is a convex curve
        between any two domain points with minimum distance of $(2^{-1} -
        2^{-1/s})$ (domain points distance), that is, for every compact
        interval $J \subset I$, where length of J is greater than, or
        equal to $(2^{-1} - 2^{-1/s})$ interval with boundary $\partial
        J$, it is true that
        $$\sup _J \left( {L - f} \right) \ge \sup _{\partial J} \left( {L -
            f} \right)$$ and $L$ is such that it is continuous, smooth, and,
        for each point $x$ of $L$, defined in terms of ninety degrees
        intercepts with the straight line between the two points of the
        function, it is true that $1 \le x \le 2^{ - 1}  + 2^{ -s}$, where
        1 corresponds to the straight line height;

        \item $f$ is convex.
    \end{itemize}
\end{definition}

In general, the class of $s$-convex functions in the second sense
would incomplete concept without a geometric interpretations for
it is behavior. Recently, Pinheiro devoted her efforts to give a
clear geometric definition for $s$-convexity in second sense. In
\cite{MRP3}  Pinheiro successfully proposed a geometric
description for $s$-convex curve, as follows:
\begin{definition}
    \label{def2.3.31}$f$ is called $s$-convex in the second sense if
    and only if one in two situations occur:

    \begin{itemize}
        \item $0 < s_1 < 1$, $f$ then belonging to $K^2_s$ , for $0 < s
        \le s_1$: The graph of f lies below (L), which is a convex curve
        between any two domain points with minimum distance of $(2^{-s} -
        2^{-1} )$ (domain points distance), that is, for every compact
        interval $J \subset I$, where length of J is greater than, or
        equal to $(2^{-s} - 2^{-1})$ interval with boundary $\partial J$,
        it is true that
        $$\sup _J \left( {L - f} \right) \ge \sup _{\partial J} \left( {L -
            f} \right)$$ and $L$ is such that it is continuous, smooth, and,
        for each point $x$ of $L$, defined in terms of ninety degrees
        intercepts with the straight line between the two points of the
        function, it is true that $1 \le x \le 2^{1 - s}$, where 1
        corresponds to the straight line height;

        \item $f$ is convex.
    \end{itemize}
\end{definition}
\noindent More geometrically, an interpretation of $s$-convex
functions is introduced as follows:
\begin{definition}
    \label{def2.3.32}$f$ is called $s$--convex, $0 < s < 1$, $f \ge
    0$, if the graph of $f$ lies below a `bent chord' $L$ between any
    two points. That is, for every compact interval $J \subset I$,
    with boundary $\partial J$, it is true that $$\mathop {\sup
    }\limits_J \left( {L - f} \right) \ge \mathop {\sup
    }\limits_{\partial J} \left( {L - f} \right).
    $$
\end{definition}

\noindent Indeed the geometric view for $s$-convex mapping of
second sense is going through which Pinheiro called it
`\emph{limiting curve}', which is going to distinguish curves that
are $s$-convex of second sense from those that are not. After
that, Pinheiro obtained how the choice of `$s$' affects the
limiting curve. In general a `limiting curve' may be described by
a \emph{bent chord} joining $f(x)$ to $f(y)$-corresponding to the
verification of the $s$-convexity property of the function $f$ in
the interval $[x, y]$-forms representing the limiting height for
the curve $f$ to be at, limit included, in case $f$ is $s$-convex.
This curve is represented by $\lambda ^s f\left( x \right) +
\left( {1 - \lambda } \right)^s f\left( y \right)$, for each
$0<s<1$.

\noindent Some properties of the limiting curve such as: maximum
height, length, and local inclination are considered in
\cite{MRP2}--\cite{MRP5}.

\begin{itemize}
    \item \textbf{Height.} The maximum of the limiting $s$-curve is
    $2^{1-s}$.

    \item \textbf{Length.}  Let $ f\left( \lambda \right) = \lambda ^s
    X + \left( {1 - \lambda } \right)^s Y $, with $X = f\left( x
    \right)$, and $Y = f\left( y \right)$. The size of the limiting
    curve from $f\left( x \right)$ to $f\left( y \right)$ is
    \begin{align*}
    L\left( \lambda  \right) &= \int_0^1 {\sqrt {1 + s^2 \lambda ^{2s
                - 2}  + s^2 \left( {1 - \lambda } \right)^{2s - 2}  - 2s^2 \lambda
            ^{s - 1} \left( {1 - \lambda } \right)^{s - 1} } d\lambda }
    \end{align*}
    which shows that how bent is the limiting curve.

    \item \textbf{Local inclination.} The local inclination of the
    limiting curve may be founded by means of the first derivative,
    consider $f\left( \lambda  \right) = \lambda ^s f\left( x \right)
    + \left( {1 - \lambda } \right)^s f\left( y \right)$, Therefore,
    the inclination is $f'\left( \lambda  \right) = s\lambda ^{s - 1}
    f\left( x \right) - s\left( {1 - \lambda } \right)^{s - 1} f\left(
    y \right)$ and varies accordingly to the value of $\lambda$.
\end{itemize}

In 1985, E. K. Godnova and V. I. Levin (see \cite{GL} or
\cite{MPF}, pp. 410-433) introduced the following class of
functions:
\begin{definition}
\label{def2}   We say that $f : I \to \mathbb{R}$ is a
Godunova-Levin function or that $f$ belongs to the class $Q\left(I
\right)$ if for all $x,y \in I$ and $t?\in (0,1)$ we have
\begin{align*}
f \left(tx+\left(1-t\right)y\right) \le
\frac{f\left(x\right)}{t}+\frac{f\left(y\right)}{1-t}.
\end{align*}
\end{definition}
In the same work, the authors proved that  all nonnegative
monotonic and nonnegative convex functions belong to this class.
For related works see \cite{DPP} and \cite{MP}.

In 1999, Pearce and Rubinov \cite{PR}, established a new type of
convex functions which is called $P$-functions.
\begin{definition}
\label{def3}  We say that  $f : I\to \mathbb{R}$ is $
 P $-function or that $f$ belongs to the class
$P\left(I \right)$ if for all $x,y \in I$ and $t?\in [0,1]$ we
have
\begin{align*}
f \left(tx+\left(1-t\right)y\right) \le
 f\left(x\right)+ f\left(y\right).
\end{align*}
\end{definition}
Indeed, $Q(I) \supseteq P (I)$ and for applications it is
important to note that $P (I)$ also consists only of nonnegative
monotonic, convex and quasi-convex functions.  A related work was
considered in \cite{DPP} and \cite{TYD}.

In 2007, Varo\v{s}anec \cite{V} introduced the class of $h$-convex
functions which generalize convex, $s$-convex, Godunova-Levin
functions and $P$-functions. Namely, the $h$-convex function is
defined as a non-negative function $f : I \to \mathbb{R}$ which
satisfies
\begin{align}
f\left( {t\alpha +\left(1-t\right)\beta} \right)\le
h\left(t\right) f\left( {\alpha} \right)+ h\left(1-t\right)
f\left( {\beta} \right),\label{h-convex}
\end{align}
where $h$ is a non-negative function, $t\in (0, 1)\subseteq J$
and $x,y \in I $, where $I$ and $J$ are real intervals such that
$(0,1) \subseteq J $. Accordingly, some properties of $h$-convex
functions were discussed in the same work of Varo\v{s}anec. For
more results; generalization, counterparts and inequalities
regarding $h$-convexity see  \cite{BV},
\cite{C}--\cite{D3},\cite{H},\cite{M}, and \cite{O}.

\section{On $h$--convex functions}

Throughout this work, $I$ and $J$ are two intervals subset of
$\left(0,\infty\right)$ such that $\left(0,1\right)\subseteq J$
and $\left[a,b\right]\subseteq I$ with $0<a<b$.
\begin{definition}
The $h$-cord joining any two points
$\left(x,f\left(x\right)\right)$ and
$\left(y,f\left(y\right)\right)$ on the graph of $f$ is defined to
be
\begin{align}
L\left(t;h\right):= \left[ {f\left( y \right) - f\left( x \right)}
\right]h\left(\frac{t-x}{y-x}\right) +  f\left( x
\right),\label{eq2.1}
\end{align}
for all $t\in [x,y]\subseteq \mathcal{I}$.  In particular, if
$h(t)=t$ then we obtain the well known form of chord, which is
\begin{align*}
L\left(t;t\right):=\frac{f\left( y \right) - f\left( x
\right)}{y-x} \left( {t-x} \right)+ f\left( x \right).
\end{align*}
\end{definition}
It's worth to mention that, if $h\left(0\right)=0$ and
$h\left(1\right)=1$, then $L\left(x;h\right)=  f\left( x \right)$
and $L\left(y;h\right)= f\left( y \right)$, so that the $h$-cord
$L$ agrees with $f$ at endpoints $x,y$, and this true for all such
$x,y \in I$.

The $h$-convexity of a function $f : I \to \mathbb{R}$ means
geometrically that the points of the graph of $f$ are on or below
 the $h$-chord joining the endpoints
$\left(x,f\left(x\right)\right)$ and
$\left(y,f\left(y\right)\right)$ for all $x,y?\in I$, $x < y$. In
symbols, we write
\begin{align*}
f\left(t\right)\le  \left[ {f\left( y \right) - f\left( x \right)}
\right]h\left(\frac{t-x}{y-x}\right) + f\left( x
\right)=L\left(t;h\right),
\end{align*}
for any $x \le t \le y$ and $x,y\in I$.
\begin{figure}[!h]
\begin{center}
\includegraphics[angle=0,width=2.6in]{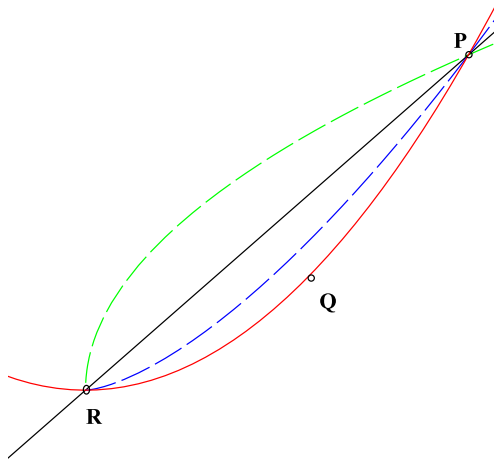}
\end{center}
\caption{The graph of $h_k(t)=t^k$, $k=\frac{1}{2}, 1 ,
\frac{3}{2}$ (green, black, blue), respectively, and $f(t)=t^2$
(red), $t\in [0,1]$.} \label{fig1}
\end{figure}

Hence, \eqref{h-convex} means geometrically that for a given three
non-collinear points $P,Q$ and $R$ on the graph of $f$ with $Q$
between $P$ and $R$ (say $P<Q<R$). Let $h$ is
super(sub)multiplicative and $h\left(\alpha\right)\ge
 (\le)\,\alpha$, for $\alpha \in \left(0,1\right)\subset J$. A
function $f$ is $h$--convex (concave) if $Q$ is on or below
(above) the $h$-chord $\widehat{PR}$ (see Figure \ref{fig1}).\\

\noindent \textbf{Caution:} In special case, for
$h_k\left(t\right) = t^k$, $t\in \left(0,1\right)$ the proposed
geometric interpretation is valid for $k \in (-1,0)\cup
(0,\infty)$. In the case that $k \le -1$ or $k=0$ the geometric
meaning is inconclusive so we exclude this case (and (and similar
cases) from our proposal above.

\begin{definition}
  Let $h: J\to \left(0,\infty\right)$ be a non-negative
function. Let $f : I \to \mathbb{R}$ be any function. We say $f$
is $h$-midconvex ($h$-midconcave) if
\begin{align*}
f\left( {\frac{{x + y}}{2}} \right) \le (\ge)\,h\left(
{\frac{1}{2}} \right)\left[ {f\left( x \right) + f\left( y
\right)} \right]
\end{align*}
 for  all $x,y \in I$.
\end{definition}
In particular, $f$ is locally $h$-midocnvex if and only if
\begin{align*}
h\left( {\frac{1}{2}} \right)\left[f\left( {x + p} \right) +
f\left( {x - p} \right)\right] - f\left( x \right) \ge 0,
\end{align*}
for  all $x\in \left(x-p,x+p\right)$, $p>0$.

A generalization of Jensen characterization of convex functions
could be stated as follows:
\begin{theorem}
Let $h: J\to \left(0,\infty\right)$ be a non-negative function
such that $h\left(\alpha\right) \ge \alpha$, for all $\alpha \in
(0,1)$. Let $f : I \to \mathbb{R}_+$ be a   nonnegative continuous
function. $f$ is $h$-convex if and only if it is $h$-midconvex;
i.e., the inequality
\begin{align*}
f\left( {\frac{{x + y}}{2}} \right) \le  h\left( {\frac{1}{2}}
\right)\left[ {f\left( x \right) + f\left( y \right)} \right],
\end{align*}
holds for  all $x,y \in I$.
\end{theorem}

\begin{proof}
The first direction follows directly by definition of
$h$-convexity. To prove the second direction, suppose on the
contrary that $f$ is not $h$-convex. Then,  there exists a
subinterval $[x, y]$ such that the graph of $f$ is not
 under the chord joining $(x, f(x))$ and $(y, f(y))$; that is,
 \begin{align*}
f\left(t\right)\ge \left[ {f\left( y \right) - f\left( x \right)}
\right]h\left(\frac{t-x}{y-x}\right) +  f\left( x \right)=L(t;h),
\end{align*}
for all such $x,y \in I \cap J$. In other words, the function
\begin{align*}
g\left( t \right) = f\left( t \right) - \left[ {f\left( y \right)
- f\left( x \right)} \right]h\left( {\frac{{t - x}}{{y - x}}}
\right) -f\left( x \right), \qquad t \in I
\end{align*}
satisfies $M = \sup \left\{ {g\left( t \right):t \in [x,y]}
\right\}
> 0$. Since   $h\left(0\right)=0$ and
$h\left(1\right)=1$, then $L\left(x;h\right)=  f\left( x \right)$
and $L\left(y;h\right)= f\left( y \right)$, so that the $h$-cord
$L$ agrees with $f$ at endpoints $x,y$. Thus, $g$ is continuous
and $g(x) = g(y) =   0$, direct computation shows that $g$ is also
mid  $h$-convex. Setting $c = \inf \left\{ {t \in [x,y]:g\left( t
\right) = M} \right\}$, then necessarily $g(c) = M$ and $c \in
(x,y)$. By the definition of $c$, for every $p > 0$ for which
$c\pm p \in (a, b)$, we have $ g\left( {c - p} \right) < g\left( c
\right)$ and $g\left( {c + p} \right) < g\left( c \right) $, so
that since $h\left(\alpha\right) \ge \alpha$, for all $\alpha \in
(0,1)$ we have
\begin{align*}
 g\left( {c - p} \right) +g\left( {c + p} \right)   <
2 g\left( {c} \right)  =\frac{1}{\frac{1}{2}} g\left( {c} \right)
\le  \frac{1}{h\left( {\frac{1}{2}} \right)} g\left( {c} \right)
\end{align*}
which contradicts the fact that $g$ is mid $h$-convex.
\end{proof}

\begin{corollary}
Let $h: J\to \left(0,\infty\right)$ be a non-negative function
such that $h\left(\alpha\right) \le \alpha$, for all $\alpha \in
(0,1)$. Let $f : I \to \mathbb{R}_+$ be a nonnegative continuous
function. $f$ is $h$-concave if and only if it is $h$-midconcave.
\end{corollary}

Sometimes we often need to know how fast limits are converging,
and this allows us to control the remainder of a given function in
a neighborhood of some point $x_0$. So that, we need to extend the
concept of continuity. Fortunately, in control theory and
numerical analysis, a function $h: J \subseteq [0,\infty) \to
[0,\infty]$ is called a control function if
\begin{enumerate}
\item $h$ is nondecreasing,

\item $\inf_{\delta > 0} h\left(\delta\right) = 0$.
\end{enumerate}

A function $f:I\to \mathbb{R}$ is $h$-continuous at $x_0$ if
$\left| f\left(x\right) - f\left(x_0\right)\right| \le h\left(|x-
x_0|\right)$, for all $x \in I$. Furthermore, a function is
continuous in   $x_0$ if it is $h$-continuous for some control
function $h$.

This approach leads us to refining the notion of continuity by
restricting the set of admissible control functions.

For a given set of control functions $\mathcal{C}$ a function is
$\mathcal{C}$-continuous if it is $h$-continuous for all $h \in
\mathcal{C}$. For example the  H\"{o}lder continuous functions of
order $\alpha \in (0,1]$  are defined by the set of control
functions
\begin{align*}
\mathcal{C}^{(\alpha)}_{H}\left(h\right) =\left\{h |
h\left(\delta\right) = H \left|\delta\right|^{\alpha},  H >
0\right\}
\end{align*}
In case $\alpha=1$, the set $\mathcal{C}^{(1)}_{H}\left(h\right)$
contains all  functions satisfying the Lipschitz condition.

\begin{theorem}\label{thm2}
Let $(0,1)\subseteq J$, $h: J\to \left(0,\infty\right)$ be a
control function which is supermultiplicative such that
$h(\alpha)\ge \alpha$ for each $\alpha\in (0,1)$. Let $I$ be a
real interval, $a,b\in \mathbb{R}$ $(a<b)$ with $a,b$ in $
I^{\circ}$ (the interior of $I$). If $f : I \to \mathbb{R}$ is
non-negative $h$-convex function on $[a,b]$, then $f$ is
$h$-continuous on $[a,b]$.
\end{theorem}

\begin{proof}
Choose $\epsilon>0$ be small enough such that
$\left(a-\epsilon,b+\epsilon\right)\subseteq I$ and let
\begin{align*}
m_{\epsilon}: = \inf \left\{ {f\left( x \right),x \in \left( {a -
\epsilon ,b + \epsilon } \right)} \right\} \qquad\text{and}\qquad
M_{\epsilon}: = \sup \left\{ {f\left( x \right),x \in \left( {a
-\epsilon ,b + \epsilon } \right)} \right\},
\end{align*}
such that $h\left(\epsilon\right)=M_{\epsilon}-m_{\epsilon}$. If
$x,y\in \left[a,b\right]$, such that
$x=y+\frac{\epsilon}{\left|y-x\right|}\left(y-x\right)$ and
$\lambda_\epsilon  =
\frac{\left|y-x\right|}{\epsilon+\left|y-x\right|}$. Then  for
$z\in \left[a-\epsilon,b+\epsilon\right]$, $y=\lambda_\epsilon
z+\left(1-\lambda_\epsilon \right)x$, we have
\begin{align*}
 f\left( y \right) = f\left( {\lambda_\epsilon  z + \left( {1 - \lambda_\epsilon  } \right)x} \right) &\le  \lambda_\epsilon  f\left( z \right) +  \left( {1 - \lambda_\epsilon  } \right)f\left( x \right) \\
  &\le  \lambda_\epsilon  \left[ {f\left( z \right) - f\left( x \right)} \right] + f\left( x
  \right)
  \le h\left( \lambda_\epsilon   \right)\left[ {f\left( z \right) - f\left( x \right)} \right] + f\left( x
  \right),
\end{align*}
which implies that $y=\lambda_\epsilon  z+\left(1-\lambda_\epsilon
\right)x$, we have
\begin{align*}
 f\left( y \right) - f\left( x \right)  \le h\left( \lambda_\epsilon   \right)\left[ {f\left( z \right) - f\left( x \right)}
 \right]
  &\le h\left( \lambda_\epsilon   \right)\left( {M_\epsilon   - m_\epsilon  } \right) \\
  &< h\left( {\frac{{\left| {y - x} \right|}}{\epsilon }} \right)\left( {M_\epsilon   - m_\epsilon  } \right) \\
  &< \frac{{h\left( {\left| {y - x} \right|} \right)}}{{h\left( \epsilon  \right)}}\left( {M_\epsilon   - m_\epsilon  } \right) \\
  &= h\left( {\left| {y - x} \right|} \right).
\end{align*}
Since this is true for any $x,y\in \left[a,b\right]$, we conclude
that $\left|f\left( y \right) - f\left( x \right)\right|\le
h\left( {\left| {y - x} \right|} \right)$, which shows that $f$ is
$h$-continuous on $\left[a,b\right]$ as desired.
\end{proof}

\noindent {\textbf {Another Proof}.} Alternatively, if one
replaces the condition $h(\alpha)+h(1-\alpha)\le 1$ for each
$\alpha\in (0,1)$ instead of $ h(\alpha)\ge \alpha$ in Theorem
\ref{thm2}. Then by repeating the same steps in the above proof,
we have
\begin{align*}
 f\left( y \right) = f\left( {\lambda_\epsilon  z + \left( {1 - \lambda_\epsilon  } \right)x} \right) &\le h\left( \lambda_\epsilon   \right)f\left( z \right) + h\left( {1 - \lambda_\epsilon  } \right)f\left( x \right) \\
  &\le h\left( \lambda_\epsilon   \right)f\left( z \right) + \left[ {1 - h\left( \lambda_\epsilon   \right)} \right]f\left( x \right)\qquad  ({\rm{since}}\,\,\,h\left( {1 - \lambda_\epsilon  } \right) \le 1 - h\left( \lambda_\epsilon   \right)) \\
  &= h\left( \lambda_\epsilon   \right)\left[ {f\left( z \right) - f\left( x \right)} \right] + f\left( x
  \right),
\end{align*}
which implies that $y=\lambda_\epsilon  z+\left(1-\lambda_\epsilon
\right)x$, we have
\begin{align*}
 f\left( y \right) - f\left( x \right)  \le h\left( \lambda_\epsilon   \right)\left[ {f\left( z \right) - f\left( x \right)}
 \right]
  &\le h\left( \lambda_\epsilon   \right)\left( {M_\epsilon   - m_\epsilon  } \right) \\
  &< h\left( {\frac{{\left| {y - x} \right|}}{\epsilon }} \right)\left( {M_\epsilon   - m_\epsilon  } \right) \\
  &< \frac{{h\left( {\left| {y - x} \right|} \right)}}{{h\left( \epsilon  \right)}}\left( {M_\epsilon   - m_\epsilon  } \right) \\
  &= h\left( {\left| {y - x} \right|} \right).
\end{align*}
Since this is true for any $x,y\in \left[a,b\right]$, we conclude
that $\left|f\left( y \right) - f\left( x \right)\right|\le
h\left( {\left| {y - x} \right|} \right)$, which shows that $f$ is
$h$-continuous on $\left[a,b\right]$. Surely, this is can be
considered as an alternative  proof  of Theorem \ref{thm2}.

 It's well known that if $f$ is twice differentiable then $f$ is convex if and only if $f''\ge0$. In a convenient  way
Pinheiro in \cite{MRP5} proposed that $f$ is an $s$-convex (in the
second sense) if and only if $f''\ge 1-2^{1-s}$. Indeed, Pinheiro
presented a ``proof" to her result, however we can say without
doubt that she introduced some good thoughts rather than formal
mathematical proof. Following the same way in \cite{MRP5} and in viewing the presented discussion in the introduction we
conjecture that:
\begin{conjecture}
Let $h: J\to \left(0,\infty\right)$ be a non-negative function
such that $h\left(\alpha\right) \ge \alpha$, for all $\alpha \in
(0,1)$, and consider $f:I\to \mathbb{R}$ be a twice differentiable
function. A function $f$ is $h$-convex if and only if
$f^{\prime\prime} \left( x \right) \ge
1-2h\left(\frac{1}{2}\right)$.
\end{conjecture}


\begin{thebibliography}{5}



\bibitem{mwo1}M.W. Alomari, M. Darus, S.S. Dragomir and U.
Kirmaci, On fractional differentiable $s$-convex functions, {\em
Jordan J. Math and Stat.}, (JJMS), {\bf3} (1) (2010), 33--42.


\bibitem{BV}M. Bombardelli and S. Varo\v{s}anec, Properties of $h$-convex
functions related to the Hermite-Hadamard-Fej\'{e}r inequalities,
{\em Compute. Math. Applica.}, {\bf58} (9) (2009), 1869--1877.

\bibitem{B1}W.W. Breckner, Stetigkeitsaussagen f\"{u}r eine Klasse
verallgemeinerter konvexer funktionen in topologischen linearen
R\"{a}umen, {\em Publ. Inst. Math.}, {\bf23} (1978), 13--20.

\bibitem{B2}W. W.
Breckner and G. Orban, Continuity properties of rationally
s-convex mappings with values in an ordered topological linear
space, Babes-Bolyai University, Cluj-Napocoi (1978).


\bibitem{B3}W.W. Breckner, H\"{o}lder-continuity of certain
generalized convex functions, {\em Optimization}, {\bf28} (1994),
201--209.

\bibitem{B4} W. W. Breckner, Rational $s$-convexity, a generalized
Jensen-convexity. Cluj-Napoca: Cluj University Press, 2011.

\bibitem{C1}S. Cobzas,  and I. Muntean,  Continuous and locally Lipschitz convex functions, {\em  Mathematica Rev. d'Anal. Num\'{e}r. et de Th\'{e}orie de I'Approx.}, Ser. Mathematica, {\bf18} (41) (1976),  41--51.

\bibitem{C2}S. Cobzas, On the Lipschitz properties of continuous convex functions, {\em Mathematlca Ret. d'Anal. Num\'{e}r. et de Th\'{e}orie de I'Approx.}, Ser. Marhemarica, {\bf 21} (44)   1979, 123?125.


\bibitem{C}M.V. Cortez, Relative strongly $h$-convex functions and integral
inequalities, {\em Appl. Math. Inf. Sci. Lett.}, {\bf4} (2)
(2016), 39--45.


\bibitem{D2} S.S. Dragomir, Inequalities of Jensen type for $h$-convex functions on linear
spaces,  {\em Math. Moravica}, {\bf 19} (1) (2015), 107--121.

\bibitem{D3} S.S. Dragomir, Inequalities of Hermite-Hadamard type for $h$-convex functions on
linear spaces, {\em Proyecciones J. Math.}, {\bf34}  (4) (2015),
323--341.

\bibitem{DPP} S.S. Dragomir, J. Pe\v{c}ari\'{c} and L.E. Persson, Some inequalities
of Hadamard type, {\em Soochow J. Math.}, {\bf21} (1995) 335--341.

\bibitem{GL}E.K. Godunova and V.I. Levin, Neravenstva dlja funkcii \v{s}irokogo
klassa, soder\v{z}a\v{s}\v{c}ego vypuklye, monotonnye i nekotorye
drugie vidy funkcii,  Vy\v{c}islitel. Mat. i. Mat. Fiz. Me?vuzov.
Sb. Nau?c. Trudov, MGPI, Moskva, 1985,   138--142.

\bibitem{H}A. H\'{a}zy, Bernstein-doetsch type results for $h$-convex functions,
{\em Math. Inequal. Appl.}, {\bf 14} (3) (2011), 499--508.

\bibitem{HM} H. Hudzik and L. Maligranda, Some remarks on $s$-convex functions,
{\em Aequationes Math.}, {\bf 48} (1994), 100--111.

\bibitem{M}M. Mat\l oka, On Hadamard's inequality for $h$-convex function on
a disk, {\em Appl. Math. Comp.},{\bf235} (2014), 118--123.


\bibitem{J} J. Musielak, Orlicz spaces and Modular spaces, Lecture Notes in Mathematics, Springer-Verlag, Berlin Heidelberg, 1983.

\bibitem{MO} W. Matuszewska and W. Orlicz, A note on the theory of $s$-normed spaces of $\psi$-integrable
functions, {\em Studia Math.}, {\bf 21}, 1981 , 107--115.

\bibitem{MP}D.S. Mitrinovi\'{c} and J. Pe\v{c}ari\'{c}, Note on a class of functions
of Godunova and Levin, {\em C. R. Math. Rep. Acad. Sci. Can.},
{\bf12} (1990), 33--36.

\bibitem{MPF}D.S. Mitrinovi\'{c}, J. Pe\v{c}ari\'{c} and A.M. Fink, Classical and New
Inequalities in Analysis, Kluwer Academic, Dordrecht, 1993.


\bibitem{NP}C.P. Niculescu, L.E. Persson, Convex Functions and Their
Applications. A Contemporary Approach, CMS Books Math., vol. 23,
Springer-Verlag, New York, 2006.






\bibitem{O}A. Olbry\'{s}, Representation theorems for
$h$-convexity, {\em J. Math. Anal. Appl}, {\bf 426} (2)(2015),
986--994.


\bibitem{PR} C.E.M. Pearce and A.M. Rubinov, $P$-functions, quasi-convex
functions and Hadamard-type inequalities, {\em J. Math. Anal.
    Appl.}, {\bf 240} (1999), 92--104.



\bibitem{P}M. Pycia, A direct proof of the $s$-H\"{o}lder continuity of Breckner
$s$-convex functions, {\em  Aequationes Math.}, {\bf 61} (1-2),
(2001), 128--130.

\bibitem{MRP1}M.R. Pinheiro, Convexity Secrets, Trafford Publishing, 2008.

\bibitem{MRP2}M.R. Pinheiro, Exploring the concept of $s$-convexity, {\em Aequationes Mathematicae},  {\bf74} (3) (2007), 201--209.

\bibitem{MRP3}M.R. Pinheiro, Hudzik and Maligranda's $s$-convexity as a local approximation to convex functions, Preprint, 2008.

\bibitem{MRP4}M.R. Pinheiro, Hudzik and Maligranda's $s$-convexity as a local approximation to convex functions II, Preprint, 2008.


\bibitem{MRP5}M.R. Pinheiro, Hudzik and Maligranda's $s$-convexity as a local approximation to convex functions III, Preprint, 2008.

\bibitem{MRP6}M.R. Pinheiro. H--H Inequality for $s$-Convex Functions, {\em Inter. J.  P. Appl.  Math.}, {\bf44} (4) (2008), 563--579.

\bibitem{RV}A. W. Roberts and D. E. Varberg,  Convex Functions, Academic
Press, New York, 1973.

\bibitem{R} S. Rolewicz, Metric Linear Spaces, 2nd ed., PWN, Warsaw, 1984.


\bibitem{TT} T. Trif, H\"{o}lder continuity of generalized convex set-valued mappings, {\em J. Math. Anal. Appl.}, {\bf255} (2001), 44--57.

\bibitem{TYD}K.-L. Tseng, G.-S. Yang and S.S. Dragomir, On quasi convex
functions and Hadamard's inequality, {\em Demonsrtatio
    Mathematics}, {\bf XLI} (2) (2008), 323--335.


\bibitem{V}S. Varo\v{s}anec,  On $h$-convexity, {\em J. Math. Anal. Appl.},  {\bf326}
(2007), 303--311.

\end{thebibliography}
\end{document}